\documentclass[12pt,reqno]{amsart}
\usepackage{pgfplots}
\usepackage{tikz}
\pgfplotsset{compat=newest}
\usepackage{color}
\makeatletter
\renewcommand*{\@cite}[2]{\fcolorbox{black}{white}{#1\if@tempswa, #2\fi}}
\renewcommand*{\@biblabel}[1]{{\fcolorbox{green}{white}{#1}}\hfill}
\makeatother
\usepackage{amsmath,amsfonts,amssymb,amsthm}
\usepackage{graphicx}
\graphicspath{ }
\usetikzlibrary{intersections}
\usetikzlibrary{patterns}
\usepackage{soul}
\usepackage[title]{appendix}
\usepackage{lipsum}
\usepackage{epstopdf}
\usepackage{pdflscape}
\usepackage{csquotes}
\usepackage{wrapfig}
\usepackage{accents}
\usepackage{adjustbox}
\usepackage{tikz-3dplot}
\usepackage{caption}
\usepackage{subcaption}
\usepackage{calligra}
\usepackage{xcolor}
\usepackage[colorlinks]{hyperref}
\hypersetup{backref=true,
            urlcolor=red,
            linkcolor=red,
            bookmarks=true,
            filecolor=black,
            citecolor=red,
            citebordercolor=green,
            filebordercolor=red,
            linkbordercolor=green
 }

\numberwithin{figure}{section}

\theoremstyle{plain}
\newtheorem{thm}{Theorem}[section]
\newtheorem{lem}[thm]{Lemma}

\theoremstyle{definition}

\numberwithin{equation}{section}

\usepackage{mathtools}

\makeatletter
\@namedef{subjclassname@2020}{%
	\textup{2020} Mathematics Subject Classification}
\makeatother

\title{On the Image of the $p$-adic Logarithm on Annuli of Principal Units}

\author{Mabud Ali Sarkar}

\address{Darjeeling Hills University, Department of Mathematics, Darjeeling-734313, India}
\email{mabudji@gmail.com}

\date{}


\begin{document}

\maketitle

\begin{abstract}
Let $K$ be a finite extension of $\mathbb{Q}_p$, and let $\mathfrak{m}_K$ be its maximal ideal. 
	 The image of the group of principal units $1+\mathfrak{m}_K$ under $p$-adic logarithm plays important role in several areas of number theory. In general, when the ramification index of $K/\mathbb{Q}_p$ is greater or equal to $p-1$, the precise description of this image is not known. For the cyclotomic extension $K=\mathbb{Q}_p(\zeta_p)$ of degree $p-1$, it was previously proved in \cite{MAS} that the image of the annulus region $(1+\mathfrak{m}_K) \setminus (1+\mathfrak{m}_K^2)$ by $p$-adic logarithm is exactly $\mathfrak{m}_K^2$.
     In this paper, we give a self-contained analytic proof  of this result based on explicit $p$-adic logarithmic expansions.
\end{abstract}
\maketitle
\section{Introduction and Motivation}\label{s1}
The $p$-adic logarithm plays a fundamental role in number theory, including in Iwasawa theory, following the pioneering works of Iwasawa~(\cite{KI1}, \cite{KI}). Knowledge of the image of the 
$p$-adic logarithm on principal units is particularly useful in explicit arithmetic computations, such as the computation of normalized $p$-adic regulators (see \cite[Theorem~2.22]{MAS}).

Let $K$ be a finite extension of $\mathbb{Q}_p$ with ring of integers $\mathcal{O}_K$ and maximal ideal $\mathfrak{m}_K$. The group of principal units $1+\mathfrak{m}_K$ is a pro-$p$ group, and the $p$-adic logarithm 
\[\log_p(1+x)=x-\frac{x^2}{2}+\frac{x^3}{3}-\cdots\]
converges for all $x \in \mathfrak{m}_K$. It is well known that the restriction of $\log_p$ induces an isomorphism
\begin{align} \label{eq1}
    \log_p: 1+\mathfrak{m}_K^r \xlongrightarrow{\simeq} \mathfrak{m}^r
\end{align}
whenever $r>\frac{e}{p-1}$, where $e$ denotes the ramification index of $K/\mathbb{Q}_p$. However, the structure of the image $\log_p(1+\mathfrak{m}_K)$ is not understood in general when the ramification index is large. In particular, $\log_p: 1+\mathfrak{m}_K \to \mathfrak{m}_K$ is not an isomorphism. 

In the special case where $K=\mathbb{Q}_p(\zeta_p)$ is the cyclotomic extension generated by a primitive $p$-th root of unity, the following was proved in earlier work \cite{MAS}:
\[\log_p\left((1+\mathfrak{m}_K) \setminus (1+\mathfrak{m}_K^2) \right)=\mathfrak{m}_K^2.\]
The aim of the present paper is to give an independent analytic proof of it, relying only on explicit $\pi$-adic expansions, following some applications.

\section{Main Result}
\begin{thm}\label{t2.1}
Let $p \geq 3 $ be a prime and consider the cyclotomic extension $K=\mathbb{Q}_p(\zeta_p)$ of $\mathbb{Q}_p$, where $\zeta_p$ is a primitive $p$th root of unity, and $\mathfrak{m}_K$ be the maximal ideal. Then the image of $(1+\mathfrak{m}_K) \setminus (1+\mathfrak{m}_K^2)$ under $p$-adic logarithm is $\mathfrak{m}_K^2$.
\end{thm}
	\begin{proof}
    We normalise the $p$-adic valuation $v$ by $v(p)=1$. Given $K=\mathbb{Q}_p(\zeta_p)$ with $\zeta_p^p=1$. It is well known that $K/\mathbb{Q}_p$ is totally ramified of degree $p-1$. By Dwork \cite{BD}, we may choose the uniformizer $\pi$ of $K$ satisfying 
    \[\pi^{p-1}=-p.\]

	 The ring of integers of $K$ is $\mathcal{O}_K=\mathbb{Z}_p[\pi]=\mathbb{Z}_p[\zeta_p]$. So the unique maximal ideal is $\mathfrak{m}_K=\pi \mathbb{Z}_p[\zeta_p]$. Let
            \[a=1+a_1\pi+a_2 \pi^2+\cdots, \cdots a_i \in \{0, 1, \cdots, p-1\}  \]
            by the Hensel expansion of the number $a \in (1+\mathfrak{m}_K)$,
            \[ a \in (1+\mathfrak{m}_K) \setminus (1+\mathfrak{m}_K^2) \Leftrightarrow a_1 \neq 0. \]
      We have
      \begin{align*}
          \log_p(a)&=\log_p(1+a_1\pi+a_2 \pi^2+\cdots) \\
          &=(a_1 \pi+a_2 \pi^2+\beta_1 \pi^3)-(\frac{a_1^2}{2} \pi^2+\beta_2 \pi^3)+\beta_3 \pi^3-\frac{a_1^p\pi^p+p \pi^{p+1}\beta_4}{p}+\beta_5 \pi^3
      \end{align*}
      where $v(\beta_i) \geq 0,~i=1,2,3,4,5$. Using the relation $\frac{\pi^p}{p}=-\pi$ of the uniformizer, and associating the first term with the $p$-th term together, we get
      \begin{align*}
          \log_p(a)&=(a_1-a_1^p) \pi +(a_2-\frac{a_1^2}{2})\pi^2+\beta_6 \pi^3
      \end{align*}
        where $v(\beta_6) \geq 0$.

        Because $a_1 \in \{0,1,2, \cdots, p-1\}$, we have $a_1-a_1^p=p \beta_7$ with $v(\beta_7) \geq 0$. Hence using $a_1-a_1^p=p \beta_7=-\pi^{p-1} \beta_7$,
        \begin{align*}
          \log_p(a)&=(a_2-\frac{a_1^2}{2})\pi^2+\beta_8 \pi^3
      \end{align*}
      for prime $p \geq 3$, where $v(\beta_8) \geq 0$.

      Thus the image of $(1+\mathfrak{m}_K)$ by $p$-adic logarithm is contained in $\mathfrak{m}_K^2$ and the image of $(1+\mathfrak{m}_K) \setminus (1+\mathfrak{m}_K^2)$ is also contained in $\mathfrak{m}_K^2$.

      We will prove now that the $p$-adic logarithm is an application of $(1+\mathfrak{m}_K) \setminus (1+\mathfrak{m}_K^2)$ onto $\mathfrak{m}_K^2$ which is a $p-1$ to 1. Take $y \in \mathfrak{m}_K^2$, we will build a (actually $p-1$)) preimage $a(y) \in (1+\mathfrak{m}_K) \setminus (1+\mathfrak{m}_K^2)$ such that $\log_p(a(y))=y$. Let $y=y_2 \pi^2+y_3 \pi^3+\cdots$
      by Hensel expansion of $y \in \mathfrak{m}_K^2$. Write
      \[y_2 \equiv a_2-\frac{a_1^2}{2} \pmod{p}, \, a_i \in \{1, \cdots, p-1 \}  \]
      which is always possible, e.g., choose $a_2$ such that $-2y_2+2a_2$ is a non-zero quadratic residue $\pmod{p}$, in this regard we note there are $\frac{p-1}{2}$ possible choices because there are $\frac{p-1}{2}$ non-zero quadratic residue modulo $p$, \cite{GH}. Once we have chosen for $a_2$ such that $-2y_2+2a_2$ is quadratic residue modulo $p$ there are two choices for $a_1 \neq 0$ namely $a_1$ or $-a_1$. So on the whole there are $p-1$ choices for the $(a_1, a_2) \in \{1, \cdots, p-1 \} \times \{0,1, \cdots, p-1\}$.

      Then choose $a_3 \in \{0,1, \cdots, p-1\}$ in such a way that
      \[ y_2 \pi^2+y_3 \pi^3 \equiv \log_p(1+a_1 \pi+a_2 \pi^2+a_3 \pi^3) \pmod{\mathfrak{m}_K^4},  \]
      this is possible because 
      \[\log_p(1+a_1 \pi+a_2 \pi^2+a_3 \pi^3)=\log_p(1+a_1 \pi+a_2 \pi^2)+ \log_p \left(1+\frac{a_3 \pi^3}{1+a_1 \pi+a_2 \pi^2} \right) \]
      By construction
      \[\log_p(1+a_1 \pi+a_2 \pi^2) \equiv y_2 \pi^2+z_3 \pi^3 \pmod{\mathfrak{m}_K^4},~v(z_3) \geq 0, \]
      where $z_3$ is the coefficient of $\pi^3$ in the second logarithm expansion of the right-hand-side.
      Hence we get
      \[y_3 \pi^3 \equiv z_3 \pi^3+a_3 \pi^3 \pmod{\mathfrak{m}_K^4} \]
      which determines $a_3 \pmod{\pi}$ and hence $\pmod{p}$ because $a_3 \in \{0,1, \cdots, p-1 \}$. By induction, the coefficients $a_i$ can be determined successively from the data $\{y_2,\dots,y_i\}$ and the previously chosen coefficients $\{a_1,\dots,a_{i-1}\}$. For a detailed and self-contained inductive determination of the coefficients, see Lemma~\ref{A1} in the Appendix.

      Thus we have proved that the image of $(1+\mathfrak{m}_K) \setminus (1+\mathfrak{m}_K^2)$ under $p$-adic logarithm is $\mathfrak{m}_K^2$ and is onto. That is,
      \[\log\left( (1+\mathfrak{m}_K) \setminus (1+\mathfrak{m}_K^2)  \right)=\mathfrak{m}_K^2. \]
	\end{proof}

\section{Applications}
\begin{thm} \label{c211}
    The index of the pro-$p$ group $\log_p(1+\mathfrak{m}_K)$ in $\mathfrak{m}_K$ is $p$, for $K=\mathbb{Q}_p(\zeta_p)$.
\end{thm}

\begin{proof}
    Note that by Theorem \ref{t2.1}
    \[ \log_p\left((1+\mathfrak{m}_K) \setminus (1+\mathfrak{m}_K^2) \right)=\mathfrak{m}_K^2. \]
    Also note that $\log(1+\mathfrak{m}_K^2)=\mathfrak{m}_K^2$, by \eqref{eq1}. 
Combining these two, we obtain
    \[ \log_p(1+\mathfrak{m}_K)=\mathfrak{m}_K^2. \]
    Thus the problem reduces to compute the index $[\mathfrak{m}_K: \mathfrak{m}_K^2]$, as we have the following filtration
    \[\mathfrak{m}_K \supset \mathfrak{m}_K^2 \supset \mathfrak{m}_K^3 \supset \cdots\]
    Whence
    \[\mathfrak{m}_K/\mathfrak{m}_K^2 \cong \mathcal{O}_K/\mathfrak{m}_K \cong \mathbb{F}_p, \]
    the residue field with $p$ elements. So $[\mathfrak{m}_K : \mathfrak{m}_K^2]=p$. The claim follows.
\end{proof}

\subsection{Application to $p$-adic regulators}

Let $F$ be a number field and $p$ a prime. The $p$-adic regulator of $F$ is defined as the covolume of the image of $\mathcal{O}_F^{\times}$ under the $p$-adic logarithm embdedding
\[ \log_p: \mathcal{O}_F^{\times} \longrightarrow \prod_{v \mid p} F_v, \]
where for each place $v \mid p$ the logarithm is normalized by
\[ \log_p: 1+\mathfrak{m}_{F_v} \longrightarrow \mathfrak{m}_{F_v}. \]
Consequently, the regulator depends on the image of the $p$-adic logarithm on principal units, and in particular, on the image of this image inside the maximal ideal. 

For a fixed place $v \mid p$, the map $\log_p$ is a homomorphism of $\mathbb{Z}_p$-modules, but in general it is not surjective. Its image is a finite index submodule of $\mathfrak{m}_{F_v}$, and its index
\[[\mathfrak{m}_{F_v}: \log_p(1+\mathfrak{m}_{F_v})]\]
measures the local distortion of volume induced by the logarithmic embedding. This index contributes directly to the normalization of the local factor in the determinant defining the $p$-adic regulator.

Assume now that for some place $v \mid p$, we have an isomorphism 
\[F_v \cong \mathbb{Q}_p(\zeta_p).\]
By Theorem \ref{c211}, we have the index
\[[\mathfrak{m}_{F_v}: \log_p(1+\mathfrak{m}_{F_v})]=p\]

This explicit description shows that, in the cyclotomic case, the local contribution to the $p$-adic regulator is completely controlled and introduces no additional ambiguity beyond the residue field size. Such local computations play an important
role in the study of normalized $p$-adic regulators.

\subsection*{Acknowledgment} The author thanks Professor Daniel Barsky for several helpful discussions.

\appendix
\section{Inductive determination of coefficients}
\begin{lem} \label{A1}
Let $K=\mathbb{Q}_p(\zeta_p)$ with uniformizer $\pi$ satisfying $\pi^{p-1}=-p$, and let
\[
y=\sum_{n\ge 2} y_n \pi^n \in \mathfrak m_K^2 .
\]
There exists an element
\[
a=1+\sum_{n\ge 1} a_n \pi^n \in (1+\mathfrak m_K)\setminus(1+\mathfrak m_K^2),
\qquad a_n\in\{0,1,\dots,p-1\},
\]
such that
\[
\log_p(a)=y .
\]
Moreover, once $a_1\neq 0$ is fixed, the coefficients $a_n$ can be determined inductively.
\end{lem}

\begin{proof}
Let us write
\[
x=\sum_{n\ge 1} a_n \pi^n \in \mathfrak m_K ,
\qquad a=1+x .
\]
Since $v(x)>0$, the $p$-adic logarithm admits the convergent expansion
\[
\log_p(1+x)=x-\frac{x^2}{2}+\frac{x^3}{3}-\cdots .
\]

Modulo $\pi^3$, we computes
\[
\log_p(1+a_1\pi+a_2\pi^2)
\equiv a_1\pi+\left(a_2-\frac{a_1^2}{2}\right)\pi^2
\pmod{\pi^3}.
\]
As shown in the proof of Theorem~\ref{t2.1}, the coefficient of $\pi$ lies in
$\mathfrak m_K^2$, hence the condition
\[
\log_p(a)\equiv y_2\pi^2 \pmod{\pi^3}
\]
reduces to
\[
a_2-\frac{a_1^2}{2}\equiv y_2 \pmod{p}.
\]
Given $a_1 \neq 0$, this congruence admits solutions, and for each such choice the
coefficient $a_2$ is uniquely determined modulo $p$.

Assume that for some $N\ge 2$ the coefficients
$a_1,\dots,a_N$ have been chosen so that
\[
\log_p\!\left(1+\sum_{n=1}^N a_n\pi^n\right)
\equiv \sum_{n=2}^N y_n\pi^n
\pmod{\pi^{N+1}}.
\]
Set
\[
a^{(N+1)}=1+\sum_{n=1}^N a_n\pi^n+a_{N+1}\pi^{N+1}.
\]
Then
\[
a^{(N+1)}=a^{(N)}\left(1+u\right),
\qquad
u=\frac{a_{N+1}\pi^{N+1}}{a^{(N)}},
\]
with $v(u)=N+1$. Using the additivity of $\log_p$ on sufficiently small elements,
we obtain
\[
\log_p(a^{(N+1)})
\equiv \log_p(a^{(N)}) + a_{N+1}\pi^{N+1}
\pmod{\pi^{N+2}}.
\]
Hence
\[
\log_p(a^{(N+1)})
\equiv \sum_{n=2}^N y_n\pi^n
+ \left(c_{N+1}+a_{N+1}\right)\pi^{N+1}
\pmod{\pi^{N+2}},
\]
where $c_{N+1}\in\mathcal O_K$ is the error term, depends only on $a_1,\dots,a_N$.
The congruence
\[
c_{N+1}+a_{N+1}\equiv y_{N+1}\pmod{p}
\]
has a unique solution for $a_{N+1}\in\{0,1,\dots,p-1\}$, which completes the
inductive step.

The resulting series converges in $\mathcal O_K$, satisfies $a_1\neq 0$, and by induction $\log_p(a)=y$.
\end{proof}


\begin{thebibliography}{20}
	\bibitem{MAS} M. A. Sarkar and A. A. Shaikh, On the image of $p$-adic logarithm on the principal units, Houston Journal of Mathematics, 50 (2024), no. 3, pp. 559–591. DoI: \url{https://arxiv.org/pdf/1904.09850}
    \bibitem{KI1} K. Iwasawa, On explicit formulas for the norm residue symbol, J. Math. Soc. Japan 20,151-165, 1968.
    \bibitem{KI} K. Iwsawa, On some modules in the theory of cyclotomic fields, J. Math. Soc. Japan 16, 42-82, 1964.
    \bibitem{GH} G. H. Hardy and E. M. Wright, An introduction to the theory of numbers, Oxford at the Clarendon Press, 4th edition, 1959.
    \bibitem{BD} Bernard Dwork, Giovanni Gerotto and Francis Sullivan, An Introduction to $G$-Functions, Annals of Mathematics Studies 133, Princeton University Press (1994).
\end{thebibliography}
\end{document}